\documentclass[11pt, reqno,fleqn]{amsart}

\usepackage{amsmath}
\usepackage{amssymb}
\usepackage{amsfonts}
\usepackage{graphicx}
\usepackage{amsthm}
\usepackage{enumerate}
\usepackage{lscape}
\usepackage{dsfont}
\usepackage{color}

\newcommand{\R}{\mathds{R}}

\newcommand{\f}{\rightarrow}                  
\newcommand{\C}{\mathds{C}}            
\newcommand{\de}{\partial}

\newtheorem{theor}{Theorem}

\DeclareMathOperator{\grad}{grad}

\begin{document}

\title[Calabi's manifold is globally symplectomorphic to $\R^{2n}$]{Calabi's inhomogeneous Einstein manifold is globally symplectomorphic to $\R^{2n}$}
\author[A. Loi, M. Zedda]{Andrea Loi, Michela Zedda}
\address{Dipartimento di Matematica e Informatica, Universit\`{a} di Cagliari,
Via Ospedale 72, 09124 Cagliari, Italy}
\email{loi@unica.it; michela.zedda@gmail.com  }
\thanks{
The first author was (partially) supported by ESF within the program ÒContact and Symplectic TopologyÓ;
the second author was  supported by  RAS
through a grant financed with the ``Sardinia PO FSE 2007-2013'' funds and 
provided according to the L.R. $7/2007$.}
\date{}

\begin{abstract}
We construct explicit global symplectic coordinates for the Calabi's inhomogeneous K\"ahler--Einstein metric on tubular domains.
\end{abstract}

\maketitle

\section{Introduction}
Let $\omega$ be  a K\"ahler form on a $n$-dimensional complex manifold  $M$   diffeomorphic to $\R^{2n}\simeq\C^n$.
One basic and fundamental question from the symplectic point of view is to understand when $(M,\omega)$ admits global symplectic coordinates, i.e. when there exists a global diffeomorphism $\Psi\!: M\f \R^{2n}$ such that $\Psi^*\omega_0=\omega$, where $\omega_0=\sum_{j=1}^ndx_j\wedge dy_j$ is the standard symplectic form on $\R^{2n}$ (the existence of a local symplectic diffeomorphism is guaranteed by the celebrated Darboux Theorem). In general the previous question has a negative answer after Gromov's discovery \cite{Gromov} of the existence of exotic symplectic structures on $\R^{2n}$ (see also \cite{Bates90} for an explicit construction of a $4$-dimensional symplectic manifold diffeomorphic to $\R^4$ which cannot be symplectically embedded in $(\R^4, \omega_0)$). Therefore it is natural to look for sufficient conditions, related to the Riemannian or to the complex structure of the manifold involved, which assure the existence of global symplectic coordinates.  
D. McDuff \cite{mcduff} proved a global version of Darboux Theorem for complete and simply-connected K\"ahler manifolds with non positive sectional curvature and in \cite{DiScalaLoi08} the first author and A. Di Scala provide a construction of a global symplectomorphism from bounded symmetric domains equipped with the Bergman metric to $(\R^{2n},\omega_0)$, using the theory of Jordan triple systems. Constructions of explicit  global symplectic coordinates on some complex domains (e.g. Reinhardt domains or Lebrun's Ricci--flat metric on $\C^2$)  is given in  \cite{CuccuLoi06} and \cite{LoiZuddas08}.

In this paper we construct  explicit global symplectic coordinates for the Calabi's inhomogeneous K\"ahler--Einstein form $\omega$ on the complex tubular domains $M=\frac{1}{2}D_a\oplus i\R^n\subset\C^n$, $n\geq 2$, 
where $D_a\subset {\R}^n$ is the open ball  of $\R^n$ centered at the origin and of radius $a$. 
 Our  main result is  the following result (see next section for details):
 \begin{theor}\label{mainteor}
For all $n\geq 2$, the K\"ahler manifold  $(M , \omega )$ is globally  symplectomorphic to $(\R^{2n},\omega_0)$ via  the map:
\begin{equation}\label{symplectomorphism}
\Phi\!:M \f\R^n\oplus i\R^n\simeq\R^{2n},\ (x,y)\mapsto \left(\grad f,y\right),
\end{equation}
where $f\!:D_a\f \R, x=(x_1, \dots x_n)\mapsto f(x)$ is a  K\"ahler potential   for $\omega$,  i.e.  $\omega=\frac{i}{2}\partial\bar\partial f$, and
$\grad f=(\frac{\partial f}{\partial x_1}, \dots ,\frac{\partial f}{\partial x_n})$.
\end{theor}

Notice that  in \cite[p. 23]{Cal} Calabi  provides
an explicit formula for the curvature tensor of $(M, g)$ (he needs this  formula to show that the metric $g$
associated to the K\"ahler form $\omega$ is not locally  homogeneous).
On the other hand  it seems a difficult task to compute the  sign of the sectional curvature of $g$ using Calabi's formula. Consequently, it is not clear if $g$
satisfies or not the assumptions of McDuff's theorem, namely  if its  sectional curvature is  nonpositive. Nevertheless, it is worth pointing out  that the proof of McDuff's result is  telling us that there exist global symplectic coordinates, but it is not giving any criterium to compute explicitly  them as we did in   Theorem \ref{mainteor}.  

We finally remark that our result should be used to give an explicit description 
of all  Langragian submanifolds of $(M, \omega)$  which  have classically played an important role in symplectic geometry.

\section{Calabi's metric and the proof of Theorem \ref{mainteor}}
Consider the complex tubular domain $M= \frac{1}{2}D_a\oplus i\R^n\subset\C^n$, as in the introduction.
Let $g$ be the metric on  $M\subset\C^n$ whose associated K\"ahler form is given by:
\begin{equation}
\omega=\frac{i}{2}\de\bar\de f(z_1+\bar z_1, \dots,z_n+\bar z_n) ,
\end{equation}
where $f\!:D_a\f \R$ is a radial function $f(x_1,\dots,x_n)=Y(r)$, for $r=(\sum_{j=1}^nx_j^2)^{1/2}$, for $x_j=(z_j+\bar z_j)/2$, $y_j=(z_j-\bar z_j)/2i$, that satisfies the differential equation:
\begin{equation}\label{basic}
(Y'/r)^{n-1}Y''=e^Y,
\end{equation}
with initial conditions:
\begin{equation}\label{incond}
Y'(0)=0,\ Y''(0)=e^{Y(0)/n}.
\end{equation}
In \cite{Cal}, Calabi proved that the  K\"ahler metric $g$ so defined is smooth, Einstein,  complete and not locally homogeneous.
This was indeed the first example of  such a metric.
 The reader is also referred to \cite{Wolf} for an alternative and easier proof of the fact that this metric is complete but not locally  homogeneous.

\begin{proof}[Proof of Theorem \ref{mainteor}]
Let us prove first that the map $\Phi$ given by (\ref{symplectomorphism}) satisfies $\Phi^*\omega_0=\omega$. In order to simplify the notation we write $\de f/\de x_j=f_j$ and $\de^2 f/\de x_j\de x_k=f_{jk}$. The pull-back of $\omega_0$ through $\Phi$ reads:
\begin{equation}
\Phi^*\omega_0=\sum_{j=1}^n df_j\wedge dy_j=\sum_{j,k=1}^n f_{jk}\,dx_k \wedge dy_j=\frac i2 \sum_{j,k=1}^nf_{jk}\, dz_j\wedge d\bar z_k,\nonumber
\end{equation}
thus the desired identity follows by:
$$\omega=\frac i2\de\bar\de f(z_1+\bar z_1,\dots,z_n+\bar z_n)=\frac i2 \sum_{j,k=1}^nf_{jk}\, dz_j\wedge d\bar z_k.$$
Observe now that   since $\omega$ and $\omega_0$ are non-degenerate it follows by the inverse function theorem  that $\Phi$ is a local diffeomorphism. 
In order to conclude the proof it is then enough to verify  that $\Phi$ is a proper map, from which it follows it is a covering map and hence a global diffeomorphism. In our situation this is equivalent to: 
 \begin{equation}
 \lim_{(x, y)\f \de M }\Phi(x, y)=\infty\nonumber
 \end{equation}
or equivalently:
\begin{equation}\label{limitcond}
\lim_{x\f \de D_a} ||\grad f (x)||=\infty.\nonumber
\end{equation}
This readily  follows by
$f_j(x)=\frac{x_j}{r}Y'(r)$ and the fact that $Y'(r)$ tends to infinity as $r\f a$ (see \cite[p. 21]{Cal}).
\end{proof}

\end{document}